\documentclass[12pt,a4paper]{amsart}
\usepackage{color,amsmath,amssymb,latexsym,amsfonts,graphicx}
\usepackage[parfill]{parskip}

\normalsize \evensidemargin=0pt \oddsidemargin=0pt \hoffset=6pt
\newtheorem{theorem}{Theorem}
\newtheorem{corollary}[theorem]{Corollary}

\newtheorem{lemma}[theorem]{Lemma}
\newtheorem{proposition}[theorem]{Proposition}
\newtheorem{remark}[theorem]{Remark}

\def\bds{\begin{displaystyle}}
\def\eds{\end{displaystyle}}

\def\1{{\mathchoice {1\mskip-4mu\mathrm l}      
{1\mskip-4mu\mathrm l}
{1\mskip-4.5mu\mathrm l} {1\mskip-5mu\mathrm l}}}
\newcommand{\floor}[1]{\left\lfloor #1 \right\rfloor}

\def\rd#1{\textcolor[rgb]{1.00,0.00,0.00}{#1}}

\begin{document}

\title{Bootstrap Random Walks}
\author{Andrea Collevecchio}
\address{Andrea Collevecchio, School of Mathematical Sciences, Monash University}
\author{Kais Hamza}
\address{Kais Hamza, School of Mathematical Sciences, Monash University}
\author{Meng Shi}
\address{Meng Shi, School of Mathematical Sciences, Monash University}

\begin{abstract}
Consider a one dimensional simple random walk $X=(X_n)_{n\geq0}$. We form a new simple symmetric random walk $Y=(Y_n)_{n\geq0}$ by taking sums of products of the increments of $X$ and study the two-dimensional walk $(X,Y)=((X_n,Y_n))_{n\geq0}$. We show that it is recurrent and when suitably normalised converges to a two-dimensional Brownian motion with independent components; this independence occurs despite the functional dependence between the pre-limit processes. The process of recycling increments in this way is repeated and a multi-dimensional analog of this limit theorem together with a transience result are obtained. The construction and results are extended to include the case where the increments take values in a finite set (not necessarily $\{-1,+1\}$).
\end{abstract}


\maketitle

\markboth{{\normalsize\sc Collevecchio , Hamza \& Shi}}{{\normalsize\sc Bootstrap Random Walks}}

\section{Introduction}\label{intro}

Consider a symmetric simple random walk
$$X_n=\sum_{k=1}^n\xi_k,\ n\geq1,\mbox{ and }X_0=0,$$
where $\xi_1,\xi_2,\ldots$ are independent and identically distributed random variables with
$$\mathbb{P}(\xi_1=-1) = \mathbb{P}(\xi_1=+1) = \frac12.$$
It is easy to see that the sequence
$$\eta_n = \prod_{k=1}^n\xi_k,\ n\geq1,$$
is made up of independent and identically distributed random variables taking values $\pm1$ with equal probability.
It immediately follows that
$$Y_n=\sum_{k=1}^n\eta_k,\ n\geq1\mbox{ and }Y_0=0$$
is also a symmetric simple random walk; that is
\begin{equation}
(Y_n)_{n\geq0} \stackrel{d}{=} (X_n)_{n\geq0}.\label{eqdist}
\end{equation}
We refer to the process of constructing $(Y_n)_n$ from $(X_n)_n$ -- that is of ``recycling'' the increments of the latter to form those of the former -- as bootstrapping.

While \eqref{eqdist} is immediately clear, what may be less understood is the behaviour of the two-dimensional process $W_n=(X_n,Y_n)$.

It is worth emphasising at this point in time that the filtrations generated by the two processes $(X_n)_{n\geq0}$ and $(Y_n)_{n\geq0}$ are identical:
$$\eta_n = \prod_{k=1}^n\xi_k\mbox{ and }\xi_n = \frac{\eta_n}{\eta_{n-1}} = \eta_{n-1}\eta_n.$$
This strong (functional) dependence is however entirely lost at infinity. More precisely, we establish that the process $(W_n)_{n\geq0}$ suitably normalised converges (weakly) to a two-dimensional Brownian motion (with independent components).
The process of taking partial products and their partial sums can then be iterated yielding a higher dimensional version of this result. Again, despite the functional dependence between the components of the pre-limit processes, the limiting process is a multidimensional Brownian motion (with independent components).

In this paper, we take a further generalising step, one that drops the requirement that $\xi_n\in\{-1,+1\}$. Instead, we allow $\xi_n$ to take values in a finite set $\mathbb{U}=\{u_0,u_1,\dots,u_{p-1}\}\subset\mathbb{R}$ and propose a general method for defining $\eta_n$ and all other iterates in such way that all partial-sum processes are identical in distribution to $(X_n)_n$. Here again, the strong dependence in the joint process is lost at infinity and the limiting process is a multidimensional Brownian motion (with independent components). The functional central limit theorem in this generalised form is presented in Section \ref{clt}.

We also briefly discuss a connection with cellular automata (see Section \ref{setup}).


The pre-limit process $W_n$ in itself is worth looking at and we present some of its properties in Section \ref{2dim}. Section \ref{setup} deals with the model setup and presents a number of basic properties including a rather precise formulation of the iterates, of any order. A number of combinatorial proofs are given in Section \ref{proofs}.

\section{Simple two and three-dimensional bootstrap walks}\label{2dim}

Let $(\xi_n)_{n\geq 1}$ be a sequence of independent and identically distributed random variables such that $\xi_i=\pm1$ with equal probability. Define $(X_n)_n$, $(\eta_n)_n$, $(Y_n)_n$ and $(W_n)_n$ as per Section \ref{intro}. We summarise our observations so far in the following proposition.

\begin{proposition}
\begin{enumerate}
\item $(\eta_n)_{n\geq0} \stackrel{d}{=} (\xi_n)_{n\geq0}$;
\item $(Y_n)_{n\geq0} \stackrel{d}{=} (X_n)_{n\geq0}$;
\item $Y_{n+1} = Y_n + (-1)^{\frac{n-X_n}2}\xi_{n+1}$ and $X_{n+1} = X_n + (-1)^{\frac{n-X_n}2}\eta_{n+1}$;
\item $(W_n)_n$ is a time-inhomegeneous Markov process;
\item $(W_{4n})_n$ is a time-homegeneous Markov process.
\end{enumerate}
\end{proposition}

\begin{remark}
As the purpose of this paper is to study the joint behaviour of random walks that are identical in law and constructed entirely by recycling the increments of one of them, the assumption that $\xi_1,\xi_2,\ldots$ (or for that matter $\eta_1,\eta_2,\ldots$) are uniformly distributed (over $\{-1,+1\}$) is crucial. Indeed, suppose $\mathbb{P}(\xi_n=1)=p$. Then,
$$\mathbb{P}(\eta_n=1)=\frac12(1+(2p-1)^n)$$
which shows that $\eta_n\stackrel{d}{=}\xi_n$ if and only if $p=1/2$. In fact, writing $\overline{\varepsilon}$ for $(\varepsilon+1)/2$ whenever $\varepsilon\in\{-1,+1\}$,
the law of $\xi_n$ can be written as $\mathbb{P}(\xi_n=\varepsilon)=p^{\overline{\varepsilon}}(1-p)^{1-\overline{\varepsilon}}$ and, for any sequence $\varepsilon_1,\ldots,\varepsilon_n\in\{-1,+1\}$,
\begin{eqnarray*}
\mathbb{P}(\eta_1=\varepsilon_1,\eta_2=\varepsilon_2,\ldots,\eta_n=\varepsilon_n) & = & \mathbb{P}(\xi_1=\varepsilon_1,\xi_2=\varepsilon_1\varepsilon_2,\ldots,\xi_n=\varepsilon_{n-1}\varepsilon_n)\\
& = & p^m(1-p)^{n-m}
\end{eqnarray*}
where $m=\sum_{k=1}^n\overline{\varepsilon_{k-1}\varepsilon_k}$ ($\varepsilon_0=1$). Again, we see that $(\eta_1,\ldots,\eta_n)\stackrel{d}{=}(\xi_1,\ldots,\xi_n)$
if and only if $p=1/2$. In other words, when $p\neq1/2$ the marginal distributions are not maintained in the recycled sequence, and the independence is lost.
\end{remark}

Next, we obtain the distribution of $W_n$ and highlight that, unlike the two-dimensional simple random walk\footnote{By this we mean the process whose components are independent simple random walks.}, its support is not square in shape. In fact it has a triangle-like shape as shown in Figure \ref{SuppW12}.

\begin{figure}[h]
\begin{center}
\includegraphics[height=2in]{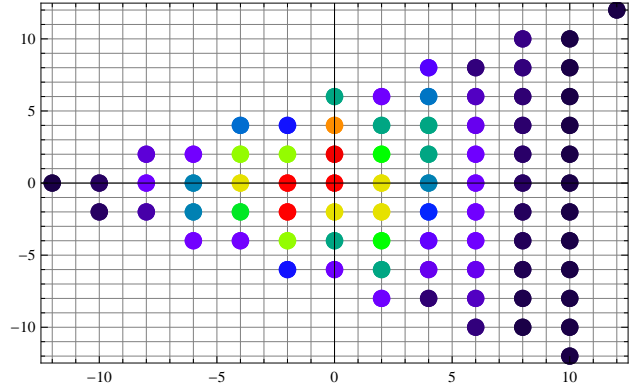}
\end{center}
\caption{Distribution of $W_{12}$ (in the black to red spectrum).}\label{SuppW12}
\end{figure}

\begin{theorem}\label{2dprob}
Assume that $n$, $k$ and $l$ are integers such that $|k|\leq n$, $|l|\leq n$ and, $n+k$ and $n+l$ are even.

If $n-k=0\mod4$ and $|2l|\leq n+k$,
$$\mathbb{P}(X_n=k,Y_n=l) = 2^{-n}\binom{\frac{n+l}{2}}{\frac{n+k+2l}{4}}\binom{\frac{n-l-2}{2}}{\frac{n+k-2l}{4}}.$$
If $n-k=2\mod4$ and $|2l+2|\leq n+k$,
$$\mathbb{P}(X_n=k,Y_n=l) = 2^{-n}\binom{\frac{n+l}{2}}{\frac{n+k+2l+2}{4}}\binom{\frac{n-l-2}{2}}{\frac{n+k-2l-2}{4}}.$$
In all other cases $\mathbb{P}(X_n=k,Y_n=l) = 0$.
\end{theorem}

Of particular interest are the probabilities of return to the origin.
\begin{corollary}
\begin{enumerate}
\item $\mathbb{P}\left(W_{4n}=0\right)=2^{-4n}\binom{2n-1}{n}\binom{2n}{n}$;
\item $\mathbb{P}\left(W_{4n+2}=0\right)=2^{-(4n+2)}\binom{2n+1}{n+1}\binom{2n}{n}$;
\item $\mathbb{P}(W_{2n}=0)\sim\frac{1}{2\pi n}$.
\end{enumerate}
\end{corollary}

The recurrence of $(W_{4n})_n$ and therefore that of $(W_n)_n$ now follow immediately.

\begin{theorem}
$(W_n)_n$ is recurrent; it will revisit the origin infinitely often.
\end{theorem}

The next natural step is to repeat the process of bootstrapping by forming successive products. We shall look at this setup in substantial generality in Section \ref{setup}. Here, we limit ourselves to the three-dimensional random walk $(X_n,Y_n,Z_n)$, also denoted $W_n$, where
$$Z_n=\sum_{k=1}^n\zeta_k,\ n\geq1,\ Z_0=0\mbox{ and }\zeta_k = \prod_{j=1}^k\eta_j = \prod_{j=1}^k\prod_{i=1}^j\xi_i,$$
and ask essentially the same questions we just answered in the two-dimensional setting.

\begin{theorem}\label{3dprob}
The following results hold for the three-dimensional bootstrap random walk $(W_n)_n$.
\begin{enumerate}
\item For any $n\geq2$,
$$\mathbb{P}\left(W_{4n}=0\right)=2^{-4n}\sum_{k=0}^{n-2} \binom{n -1}{k}\binom{n}{k+1}\binom{n}{k+1}\binom{n-1}{k+1}$$
and
$$\mathbb{P}\left(W_{4n+2}=0\right)=2^{-(4n+2)}\sum_{k=1}^n \binom{n+1}{k} \binom{n-1}{k-1}\binom{n}{k}\binom{n+1}{k+1}.$$
\item $\mathbb{P}(W_{2n}=0)=O(n^{\alpha-2})$, for any $\alpha\in(1/2,1)$.
\item $(W_n)_n$ is transient; it will visit the origin finitely often.
\end{enumerate}
\end{theorem}

\begin{remark}
It immediately follows from the previous result that any multi-dimensional random walk whose three-dimensional projection is $W_n$,
is transient. The $(K+1)$-dimensional random walk introduced in Section \ref{setup} is such an example.
\end{remark}

In summary, two and three-dimensional bootstrap random walks share many of the characteristics of simple random walks. The next limit theorem
reinforces this observation. It states that these random walks appropriately normalised converge, as simple random walks do, to independent Brownian motions.

\begin{theorem}
Let $\bds\mathfrak{W}_n(t) = \frac{1}{\sqrt{n}}W_{\lfloor nt\rfloor}\eds$, $t\in[0,1]$. $\mathfrak{W}_n$ converges weakly to a three-dimensional Brownian motion (with independent components).
\end{theorem}

The proofs of the above statements are given in Section \ref{clt}.

\section{The model setup}\label{setup}

In this section we generalise the previous setting in two directions. First, we allow the random variable $\xi_n$ to take values in any finite set
$\mathcal{U}=\{u_0,u_1,\dots,u_{p-1}\}\subset\mathbb{R}$. Then, we iterate the process of bootstrapping an arbitrary number of times.
We shall assume that $p$ is a prime number. The case when $p$ is not prime is discussed at the end of this section.

The first obstacle we have to overcome stems from the fact that in general, if $x_1,x_2\in\mathcal{U}$, $x_1x_2\not\in\mathcal{U}$. An easy way to get over this hurdle is to define a map $\mathcal{U}\times\mathcal{U}\longrightarrow\mathcal{U}$ that will replace the usual product. In other words, we define an operation $\otimes$ on $\mathcal{U}$. To extend the mapping to higher dimensions while
maintaining the flexibility afforded by the usual multiplication (associativity and commutativity), we assume that $(\mathcal{U},\otimes)$ is an Abelian group.

We shall write $e$ for the unit element of $(\mathcal{U},\otimes)$ and $u^{\otimes n}$ for the $n$th power of $u\in\mathcal{U}$. Using the Lagrange Theorem that states that the cardinality of a subgroup must divide the cardinality of the group, we immediately obtain the following results.

\begin{proposition}
Let $u\in\mathcal{U}\setminus\{e\}$.
\begin{enumerate}
\item $p$ is the smallest positive integer such that $u^{\otimes  p}=e$.
\item $u^{\otimes n}=u^{\otimes m}$ if and only if $n=m\mod p$.
\item $\mathcal{U}$ is cyclic; that is $\mathcal{U}=\left<u\right>\doteq\{e, u, u^{\otimes2},..., u^{\otimes(p-1)}\}$.
\end{enumerate}
\end{proposition}

Next we introduce the forward bootstrap operator on the set $\mathfrak{U}$ of sequences in $\mathcal{U}$
$$\begin{array}{lccc}
\Delta: & \mathfrak{U} & \longrightarrow & \mathfrak{U}\\
& (x_n)_n & \longrightarrow & \big(\bigotimes_{\ell=1}^nx_\ell\big)_n
\end{array}$$
and its inverse, the backward bootstrap operator
$$\begin{array}{lccc}
\Delta^{-1}: & \mathfrak{U} & \longrightarrow & \mathfrak{U}\\
& (y_n)_n & \longrightarrow & \big(y_{n-1}^{\otimes(p-1)}\otimes y_n\big)_n
\end{array}$$

These mappings can be iterated to define the operator $\Delta^K$, for any positive integer $K$. It is then easy to see that if $x_{\bullet}=(x_n)_n\in\mathfrak{U}$ and $y_{K,\bullet}=\Delta^K(x_\bullet)$,
$$y_{K,n} = \bigotimes_{\ell=1}^nx_{n-\ell+1}^{\otimes\nu_{K,\ell}},$$
for some array $\nu_{K,\ell}\in\{0,1,\ldots,p-1\}$.

One can more generally define the mapping $\Delta^K$ for any integer $K$. Furthermore, for any integers $K$ and $J$, we have
$$y_{K,\bullet}=\Delta^{K-J}(y_{J,\bullet}).$$
In particular,
\begin{equation}
y_{K,\bullet}=\Delta^K(x_\bullet)\mbox{ and }\Delta^{-K}(y_{K,\bullet})=x_\bullet.\label{Delta}
\end{equation}

\begin{figure}[h]
\begin{center}
\includegraphics[height=2.5in]{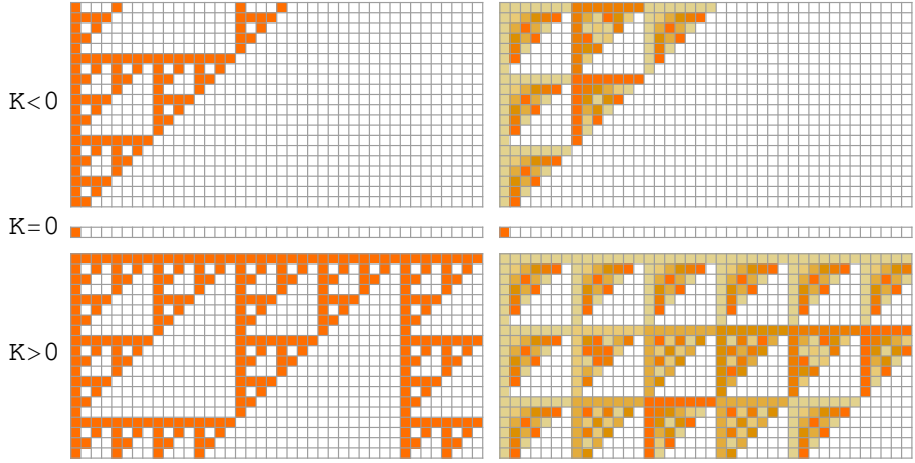}
\end{center}
\caption{The array $\nu_{K,\ell}$ when $p=2$ (left) and $p=7$ (right)}\label{2diterates}
\end{figure}

\begin{proposition}
The array $\nu_{K,n}$, $K\in\mathbb{Z}$, $n\in\mathbb{N}^*$, satisfies the following (defining) properties:
\begin{enumerate}
\item $\nu_{0,1}=1$ and $\nu_{0,n}=0$ for $n\geq 2$;
\item $\nu_{K,1}=1$ for any $K$;
\item $\nu_{K+1,n+1}=\nu_{K+1,n}+\nu_{K,n+1} \mod p$, for any $K$ and any $n\geq1$.
\end{enumerate}
It follows that, for any $K$ and any $n\geq1$,
$$\nu_{K,n}= {n+K-2 \choose n-1} = \frac{(K+n-2)(K+n-3)\ldots K}{(n-1)!}\mod p.$$
\end{proposition}
\begin{proof}
Let $x_\bullet=(x_n)_n\in\mathfrak{U}$, $y_\bullet=\Delta(x_\bullet)$ and $y_{K,\bullet}=\Delta^K(x_\bullet)$.
We deduce (1) and (2) from the facts that $y_{0,\bullet}=x_\bullet$ and $x_1=y_{K,1}=x_1^{\otimes\nu_{K,1}}$, respectively.
(3) follows from:
\begin{eqnarray*}
\bigotimes_{\ell=1}^nx_{n-\ell+1}^{\otimes\nu_{K+1,\ell}} & = & y_{K+1,n}\ =\ \Delta^K(y_\bullet)_n\ = \ \bigotimes_{\ell=1}^ny_{n-\ell+1}^{\otimes\nu_{K,\ell}}\ =\ \bigotimes_{\ell=1}^n\left(\bigotimes_{k=1}^{n-\ell+1}x_k\right)^{\otimes\nu_{K,\ell}}\\
& = & \bigotimes_{\ell=1}^n\bigotimes_{k=1}^{n-\ell+1}x_k^{\otimes\nu_{K,\ell}}\ =\ \bigotimes_{k=1}^n\bigotimes_{\ell=1}^{n-k+1}x_k^{\otimes\nu_{K,\ell}}\ =\ \bigotimes_{k=1}^nx_k^{\otimes(\nu_{K,1}+\ldots+\nu_{K,n-k+1})}\\
& = & \bigotimes_{\ell=1}^nx_{n-\ell+1}^{\otimes(\nu_{K,1}+\ldots+\nu_{K,\ell})}
\end{eqnarray*}
\end{proof}

As in the simple random walk setting, we assume that the random variables $\xi_1,\xi_2,\ldots$ are independent and have a common uniform distribution on $\mathcal{U}$:
$$\mathbb{P}(\xi_n=u_i)=\frac{1}{p},\quad 0\leq i\leq p-1.$$

Further, we define recursively $\eta_{K,n}$ as
$$\eta_{1,n}=\bigotimes_{\ell=1}^n \xi_\ell, \ \ \ \eta_{K,n}=\bigotimes_{\ell=1}^n \eta_{K-1,\ell},$$
and write for simplicity $\eta_n$ for $\eta_{1,n}$.

Define $(\mathcal{F}_n)_{n\geq 1}$ to be the natural filtration generated by the sequence $(\xi_n)_{n\geq1}$. From \eqref{Delta} we get that, for any $K$, $$\mathcal{F}_n=\sigma(\xi_1,\cdots,\xi_n)=\sigma(\eta_{K,1},\cdots,\eta_{K,n}).$$

\begin{proposition}
For any given $K$, $(\eta_{K,n})_n$ and $(\xi_n)_n$ have the same distribution. In particular, $\eta_{K,n}$ is uniform over $\mathcal{U}$ and is independent of $\mathcal{F}_{n-1}$.
\end{proposition}
\begin{proof} It is, of course, sufficient to prove the result for $K=1$, which we establish with the aid of the backward bootstrap operator:
\begin{eqnarray*}
\lefteqn{\mathbb{P}\left(\eta_1=y_1,\eta_2=y_2,\dots,\eta_n=y_n\right)}\\
& = & \mathbb{P}\left(\xi_1=y_1,\xi_2=y_1^{\otimes(p-1)}\otimes y_2,\dots,\xi_n=y_{n-1}^{\otimes(p-1)}\otimes y_n\right)
\ =\ (1/p)^n.
\end{eqnarray*}
\end{proof}

\begin{remark}
Note that not only do we have $(\eta_n)_n\stackrel{d}{=}(\xi_n)_n$ but for any sequence of integers $m_n\neq0\mod p$, $(\xi_n^{\otimes m_n})_n\stackrel{d}{=}(\xi_n)_n$ and therefore $\left(\bigotimes_{k=1}^n\xi_k^{\otimes m_k}\right)_{n\geq0} \stackrel{d}{=} (\xi_n)_{n\geq0}$. Indeed, let $y_1,\ldots,y_n\in\mathcal{U}$. Fix $u\in\mathcal{U}$, $u\neq e$. Then we can write $y_1,\ldots,y_n$ as $u^{\otimes j_1},\ldots,u^{\otimes j_n}$ and using Lemma \ref{bezout}, we have
\begin{eqnarray*}
\lefteqn{\mathbb{P}\left(\xi_1^{\otimes m_1}=u^{\otimes j_1},\ldots,\xi_n^{\otimes m_n}=u^{\otimes j_n}\right)}\\
& = & \mathbb{P}\left(\xi_1^{\otimes m_1}=u^{\otimes(j_1+l_1p)},\ldots,\xi_n^{\otimes m_n}=u^{\otimes(j_n+l_np)}\right)\\
& = & \mathbb{P}\left(\xi_1^{\otimes m_1}=u^{\otimes(k_1m_1)},\ldots,\xi_n^{\otimes m_n}=u^{\otimes(k_nm_n)}\right)\\
& = & \mathbb{P}\left(\xi_1=u^{\otimes k_1},\ldots,\xi_n=u^{\otimes k_n}\right)\ =\ (1/p)^n.
\end{eqnarray*}
\end{remark}

\begin{lemma}\label{bezout}
For any integer $m\neq0 \mod p$ and $0\leq j<p$, there exists a pair of
integers $(k,\ell)$ such that $km=j+\ell p$.
\end{lemma}
\begin{proof}
Since $p$ is prime, $\gcd(m,p)=1$. By B\'ezout's identity, there exist integers $a$
and $b$ such that $am+pb=1$. Then with $k=ja$ and $\ell=-jb$, we have $km=j+\ell p$.
\end{proof}

The next proposition looks at the dependence structure of the columns in the $\eta_{K,n}$ array. It relies on the following lemma.

\begin{lemma}
Fix $x_1,\ldots,x_{n-1}$ and $x_{n+K},y_1,\ldots,y_K$ all in $\mathcal{U}$.
The system of equations in $x_n,\ldots,x_{n+K-1}$
\begin{equation}\label{system}
\bigotimes_{\ell=1}^{n+K}x_{n+K-\ell+1}^{\otimes\nu_{k,\ell}} = y_k,\ k=1,\ldots,K
\end{equation}
has a unique solution.
\end{lemma}
\begin{proof}
Write $y_{k,n} = \bigotimes_{\ell=1}^nx_{n-\ell+1}^{\otimes\nu_{k,\ell}}$ so that \eqref{system} is equivalent to
$$y_{k,n+K}=y_k,\ k=1,\ldots,K.$$
With the aid of the following representation we show that a solution exists and is unique. While we do not give an explicit expression for this solution, the mechanism to obtain it is clear. Starting from the light blue (first row) and light red (last column) cells we construct the remainder of the array. The colours are merely an indication of the steps in the construction and do not represent particular values.
\begin{figure}[h]
\begin{center}
\includegraphics[height=1.5in]{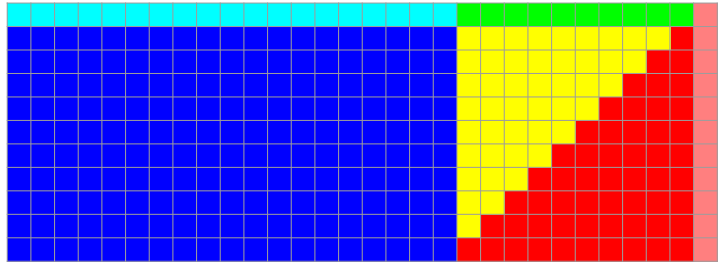}
\end{center}
\caption{The array $y_{k,\ell}$ for $0\leq k\leq K$ and $1\leq\ell\leq n+K$.}
\end{figure}

Using the fact that any two cells in a triangular array of the type
\begin{center}\begin{tabular}{|c|c|}\cline{2-2}\multicolumn{1}{c|}{}&x\\\hline y&z\\\hline\end{tabular}\end{center}
uniquely determine the third, and the fact that the first column is identical to its top cell ($y_{k,1}=y_{0,1}=x_1$), we see that we can work our way in a unique fashion from the the vector $(y_{0,n+K},\ldots,y_{K,n+K})$ (the light red cells on the right-most column) to $(y_{K,n},\ldots,y_{0,n+K})$ and other intermediate values (the cells forming the red triangular array), and from $(x_1,\ldots,x_{n-1})$ (the light blue cells on the first row) to $(y_{0,n-1},\ldots,y_{K,n-1})$ and other intermediate values (the rectangular array made up of blue cells and in particular the right-most column within it). Combining the red and blue cells, we can then work our way up through the yellow cells and arrive to a unique set of values for $(x_n,\ldots,x_{n+K-1})$ (the green cells on the top row).
\end{proof}

\begin{remark}
A by-product of the the above lemma is an interesting observation on the square matrix $(\nu_{k,\ell+1})_{1\leq k,\ell\leq K}$. Indeed, \eqref{system} can be rewritten
$$x_n^{\otimes\nu_{k,K+1}}\otimes x_{n+1}^{\otimes\nu_{k,K}}\otimes\ldots\otimes x_{n+K-1}^{\otimes\nu_{k,2}} = x_{n+K}^{\otimes(-1)}\otimes y_k\otimes\bigotimes_{\ell=K+2}^{n+K}x_{n+K-\ell+1}^{\otimes(-\nu_{k,\ell})},\
k=1,\ldots,K,$$
which has a unique solution if and only if the matrix $(\nu_{k,\ell+1})_{1\leq k,\ell\leq K}$ is non-singular. As a result we get the that, for any $K$, universally in $p$,
$$\left|
\begin{array}{cccc}
\nu_{1,2} & \nu_{1,3} & \ldots & \nu_{1,K+1}\\
\nu_{2,2} & \nu_{2,3} & \ldots & \nu_{2,K+1}\\
\vdots & \vdots & \ddots & \vdots\\
\nu_{K,2} & \nu_{K,3} & \ldots & \nu_{K,K+1}
\end{array}
\right|\neq0.$$
Choosing $p$ large enough shows the result to be true for $\nu_{k,\ell}$ replaced with the binomial coefficient ${\ell+k-2\choose\ell-1}$.
\end{remark}

\begin{proposition}\label{kappadep}
The vector $(\eta_{0,n+K},\ldots,\eta_{K,n+K})$ is uniform over $\mathcal{U}^{K+1}$ and is independent of $\mathcal{F}_{n-1}$. In particular, the random variables $\eta_{0,n+K},\ldots,\eta_{K,n+K}$ are independent.
\end{proposition}
\begin{proof}
Fix $x_1,\ldots,x_{n-1}$ and $x_{n+K},y_1,\ldots,y_K$ all in $\mathcal{U}$. Using the above lemma, we immediately get that
\begin{align*}
&\mathbb{P}(\eta_{0,n+K}=x_{n+K},\eta_{1,n+K}=y_1,\ldots,\eta_{K,n+K}=y_K|\xi_1=x_1,\ldots,\xi_{n-1}=x_{n-1})\\
&=\frac{(1/p)^{n+K}}{(1/p)^{n-1}}\ =\ (1/p)^{K+1}.
\end{align*}
We conclude the proof by observing that the above conditional probability is independent of the choices of $x_1,\ldots,x_{n-1}$ and $x_{n+K},y_1,\ldots,y_K$.
\end{proof}

The functional central limit theorem given in Section \ref{clt} relies on a detailed analysis of the relationship between the various $\eta_{K,n}$'s. Being a product of powers of $\xi_1,\ldots,\xi_n$,
\begin{equation}
\eta_{K,n}=\bigotimes_{\ell=1}^n \xi_{n-\ell+1}^{\otimes\nu_{K,\ell}}
\end{equation}
we need to identify those that are multiples of $p$. As such, the corresponding $\xi_k$'s are ``switched off'' making them independent of $\eta_{K,n}$. The following results address these very issues.

To understand the structure of $\nu_{K,n}$, we make use of the following theorem (see \cite{Lucas}, p229). Recall that the base $p$ expansion of $n$ is $n=\alpha_kp^k+\ldots+\alpha_1p+\alpha_0$ where $\alpha_0,\dots,\alpha_k\in\{0,1,\ldots,p-1\}$ are the base $p$ digits of $n$.

\begin{theorem}\label{Lucas}[Lucas]
A binomial coefficient $\binom{n}{m}$ is divisible by a prime $p$ if and only if at least one of the base $p$ digits of $m$ is greater than the corresponding digit of $n$.
\end{theorem}

The following proposition is essential to our analysis. Its proof can be found in the appendix.
\begin{proposition}\label{nu}
The following properties hold for $\nu_{K,n}$:
\begin{enumerate}
\item for $K=p^\ell$ and $1<n\leq p^\ell$, then $\nu_{K,n}=0$;
\item for $n=p^\ell$ and $1\leq K\leq p^\ell$, then $\nu_{K,n}=0$;
\item for $n=p^\ell+1$ and $1\leq K\leq p^\ell$, then $\nu_{K,n}=1 \mod p$;
\item for $1\leq K\leq p^\ell$, $\nu_{K,p^\ell-K+1}\neq 0 \mod p$.
\end{enumerate}
\end{proposition}

\begin{corollary}
Let $\omega_K=\min\{n\geq2:\nu_{K,n}\neq0\mod p\}$. Then $\omega_0=+\infty$ and, if $p^\ell$ is the smallest power of $p$ greater than or equal to a positive integer $K$, then $\omega_K\leq p^\ell-K+1$.
\end{corollary}

The final ingredient in the model setup is to define the random walks themselves:
\begin{equation}
Y_{K,n}=\sum_{i=1}^n \eta_{K,i}, \ \ \mbox{ with }Y_{K,0}=0,
\end{equation}
where $\sum$ represents the usual sum in $\mathbb{R}$ and $K\in\mathbb{Z}$. We shall maintain the notation $X_n$ for $Y_{0,n}$.

From the equality in law of the sequences $(\eta_{K,n})_n$, we immediately deduce that,
for any given $K$, $(Y_{K,n})_n$ is a random walk identical in law to $(X_n)_n$.

As our aim is to prove a central limit theorem for the $(K+1)$-dimensional random walk\footnote{Here again we abuse notations by referring to this process as $W_n$.} $W_n=(Y_{0,n},\ldots,Y_{K,n})$, a necessary step of which is the removal of its drift, we can assume without loss of generality that
$$\mathbb{E}[\xi_n]=0.$$
When this is coupled with the requirement that $\xi_n$ must have a uniform distribution (to guarantee that
the distribution of $(X_n)_n$ is preserved after bootstrapping), we obtain the following condition on the values in $\mathcal{U}$:
\begin{equation}
u_0+u_1+\dots+u_{p-1}=0.
\end{equation}

The second moment of $\xi_n$ plays an important role. We denote it by $\sigma^2$:
\begin{equation}
\sigma^2 = \mathbb{E}[\xi_n^2] = \frac1p(u_0^2+u_1^2+\dots+u_{p-1}^2).
\end{equation}

\begin{proposition}\label{etacorr}
Let $K$ and $J$ be integers and, $m$ and $n$ be positive integers.
\begin{enumerate}
\item If $m\neq n$, then $\mathbb{E}[\eta_{K,m}\eta_{J,n}]=0$.
\item If $m=n$ and $K=J$, then $\mathbb{E}[\eta_{K,m}\eta_{J,n}]=\sigma^2$.
\item If $K\neq J$ and $m=n<\omega_{|K-J|}$, then $\mathbb{E}[\eta_{K,m}\eta_{J,n}]=\sigma^2$.
\item If $K\neq J$ and $m=n\geq\omega_{|K-J|}$, then $\mathbb{E}[\eta_{K,m}\eta_{J,n}]=0$.
\end{enumerate}
It follows that, for $m\leq n$,
\begin{enumerate}\setcounter{enumi}{4}
\item $\bds\mathbb{E}[Y_{K,m}Y_{J,n}] = \min(m,\omega_{|K-J|}-1)\sigma^2\eds$.
\end{enumerate}
\end{proposition}
\begin{proof}
The first two statements follow from the identity in law:
$$(\eta_{K,m},\eta_{J,n}) = ((\Delta^{K-J}(\eta_{J,\bullet}))_m,\eta_{J,n})
\stackrel{d}{=} ((\Delta^{K-J}(\xi_\bullet))_m,\xi_n) = (\eta_{K-J,m},\xi_n),$$
where $\eta_{J,\bullet}=(\eta_{J,n})_n$ and $\xi_\bullet=(\xi_n)_n$. To show the next two statements we proceed as follows.
Suppose $K>J$. We write
\begin{eqnarray*}
\mathbb{E}[\eta_{K,m}\eta_{J,m}]\ =\ \mathbb{E}[\eta_{K-J,m}\xi_m]
& = & \mathbb{E}\left[\xi_m\left(\xi_m\otimes\bigotimes_{\ell=2}^m\xi_{m-\ell+1}^{\otimes\nu_{K-J,\ell}}\right)\right]\\
& = & \mathbb{E}\left[\xi_m\mathbb{E}\left[\left.\left(\xi_m\otimes\bigotimes_{\ell=2}^m\xi_{m-\ell+1}^{\otimes\nu_{K-J,\ell}}\right)\right|\xi_m\right]\right]
\end{eqnarray*}
then we observe that, for $m\geq\omega_{K-J}$ (so that at least one $\nu_{K-J,\ell}\neq0\mod p$) and any $u\in\mathcal{U}$, $u\otimes\bigotimes_{\ell=2}^m\xi_{m-\ell+1}^{\otimes\nu_{K-J,\ell}}$ is uniformly distributed over $\mathcal{U}$. Of course, in the case $m<\omega_{K-J}$, $\mathbb{E}[\eta_{K,m}\eta_{J,m}] = \mathbb{E}[\xi_m^2] = \sigma^2$.

Suppose, for the last statement, that $m\leq n$, then in view of (1),
$$\mathbb{E}[Y_{K,m}Y_{J,n}] = \sum_{i=1}^m\sum_{j=1}^n\mathbb{E}[\eta_{K,i}\eta_{J,j}] =
\sum_{i=1}^m\mathbb{E}[\eta_{K,i}\eta_{J,i}].$$
The result follows by application of (2), (3) and (4).
\end{proof}

We end this section with 2 remarks.

\begin{remark}
When $p$ is not prime, one can alter the distribution of $\xi$ by adding sufficiently many zeroes to make the number of possible values prime. On the one hand, these zeroes will be seen differently by the operation $\otimes$ yielding a cyclic group, on the other hand, their only impact on the the bootstrapped random walks is to slow down their evolutions. The resulting random walk is simply a ``lazy'' version of the original one.
\end{remark}

\begin{remark}
In the case $p=2$ and when one focuses on the array $\eta_{K,n}$ (i.e. the increments and not the random walks themselves), then the setup appears as an ``infinite'' memory cellular automaton, which can be reduced to a regular one (CA60) by performing a ``sliding'' of the columns (see Figure \ref{2DIteratesCA60}).
\begin{figure}[h]
\begin{center}
\includegraphics[height=2in]{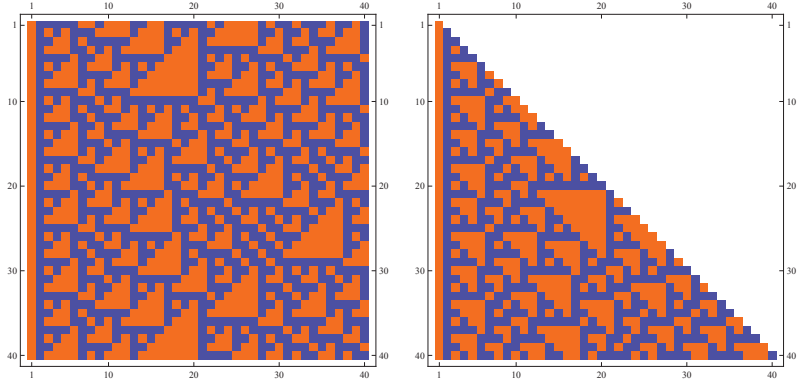}
\end{center}
\caption{The array $\eta_{K,\ell}$ (in its original form on the left) becomes after ``sliding'' (right) a cellular automaton 60 ($p=2$).}\label{2DIteratesCA60}
\end{figure}
The sole reason for this observation is for completeness as our focus is on the random walks and this connection has no bearing on our results or thinking.
\end{remark}

\section{A functional central limit theorem}\label{clt}

It is well known that $\bds\mathfrak{X}_n(t) = \frac1{\sigma\sqrt{n}}X_{\lfloor nt\rfloor}\eds$ and more generally $\bds\mathfrak{Y}_{K,n}(t) = \frac1{\sigma\sqrt{n}}Y_{K,\lfloor nt\rfloor}\eds$ converge weakly to a Brownian motion ($t\in[0,1]$). The focus of this section is the $(K+1)$-dimensional process
$\mathfrak{W}_n(t)=(\mathfrak{Y}_{0,n}(t),\ldots,\mathfrak{Y}_{K,n}(t))$, $t\in[0,1]$.

\begin{theorem}\label{mength}
For any positive integer $K$, $\mathfrak{W}_n$ converges weakly to a $(K+1)$-dimensional Brownian motion (with independent components).
\end{theorem}

\begin{proof} Using the Cram\'er-Wold device (see for example Billingsley \cite{Billingsley})
we reduce this multi-dimensional problem to a one-dimensional one. To this end we fix a normalised vector $(a_0,\ldots,a_K)\in\mathbb{R}^{K+1}$ ($\sum_{k=0}^K a_k^2=1$) and focus on the sequence of processes
$$\mathfrak{S}_n(t)=\sum_{k=0}^K a_k\mathfrak{Y}_{k,n}(t) = \frac1{\sigma\sqrt{n}}\sum_{k=0}^K a_kY_{K,\lfloor nt\rfloor}
= \frac1{\sigma\sqrt{n}}\sum_{\ell=1}^{\lfloor nt\rfloor}\sum_{k=0}^K a_k\eta_{k,\ell} = \frac1{\sigma\sqrt{n}}S_{\lfloor nt\rfloor},$$
where
$$S_n = \sum_{\ell=1}^nR_\ell\mbox{ and }R_\ell = \sum_{k=0}^K a_k\eta_{k,\ell}.$$
By Proposition \ref{etacorr}, the random variables $R_n$ are clearly uncorrelated
$$\mathbb{E}[R_mR_n] = \sum_{k,\ell=0}^K a_ka_\ell\mathbb{E}[\eta_{k,m}\eta_{\ell,n}] = 0$$
and we have
$$\mathbb{E}[R_n^2] = \sum_{k,\ell=0}^K a_ka_\ell\mathbb{E}[\eta_{k,n}\eta_{\ell,n}]
= \sigma^2\sum_{k=0}^K a_k^2 + 2\sigma^2\sum_{k=1}^K\sum_{\ell=0}^{k-1}a_ka_\ell 1_{n<\omega_{k-\ell}},$$
which, for $n\geq\max_{1\leq k\leq K}\omega_k$, reduces to
$$\mathbb{E}[R_n^2] = \sigma^2.$$
We deduce that $S_n$ is a martingale with respect to the filtration $\mathcal{F}_n$ and
\begin{equation}\label{sn2}
s_n^2 = \mathbb{E}[S_n^2] = \sum_{\ell=1}^n\mathbb{E}[R_\ell^2] = n\sigma^2+C,
\end{equation}
where $C$ is a constant that only depends on $K$ and $a_0,\ldots,a_K$ and not on $n$.

Next we introduce an intermediary process $\mathfrak{M}_n$ and use a result of Scott \cite{Scott} to show that it approaches a Brownian motion (weakly). For $t\in[0,1]$,
we set $\mathfrak{M}_n(t) = S_k/s_n$ whenever $k$ is such that $s_k^2\leq ts_n^2<s_{k+1}^2$. Then, since $S_n$ has bounded increments, to establish the weak convergence of $\mathfrak{M}_n$ to a standard Brownian motion, it is sufficient to show that $\bds\frac{1}{s_n^2}\sum_{\ell=1}^nR_\ell^2\stackrel{p}{\longrightarrow}1\eds$ (see \cite{Scott}).

In fact we shall prove a stronger result in which the convergence is almost sure. By Proposition \ref{kappadep} we know that, for any $k\in\{1,\ldots,K\}$, $(R_{k+nK})_n$ are independent random variables. If we now let $V_{k,n} = R_{k+n(K+1)}^2-\mathbb{E}[R_{k+n(K+1)}^2]$ and observe that, since the $R_n$'s are bounded random variables, so are the $V_{k,n}$'s. It follows that $\bds\sum_n\frac{\mathbb{E}[V_{k,n}^2]}{n^2}<+\infty\eds$ from which we deduce (see for example \cite{Williams} p118) that
$\bds\frac1n\sum_{\ell=0}^{n-1}(R_{k+\ell(K+1)}^2-\mathbb{E}[R_{k+\ell(K+1)}^2])\stackrel{a.s.}{\longrightarrow}0\eds$ and consequently that
$$\frac{1}{s_n^2}\sum_{\ell=0}^{n-1}R_{k+\ell(K+1)}^2\stackrel{a.s.}{\longrightarrow}1.$$
Here we have used the facts that $s_n^2\sim\sigma^2n$ and $\mathbb{E}[R_n^2]=\sigma^2$, for $n$ large enough.

Summing over $k\in\{1,\ldots,K+1\}$ yields
$$\frac{1}{s_{n(K+1)}^2}\sum_{\ell=1}^{n(K+1)}R_\ell^2 = \frac{(K+1)s_n^2}{s_{n(K+1)}^2}\ \frac1K\sum_{k=1}^K\frac{1}{s_n^2}\sum_{\ell=0}^{n-1}R_{k+\ell(K+1)}^2\stackrel{a.s.}{\longrightarrow}1.$$

It now follows that $\mathfrak{M}_n$ converges weakly to a Brownian motion and so does $(s_n/(\sigma\sqrt{n}))\mathfrak{M}_n$.

The final step is to establish that the processes $\mathfrak{S}_n$ and $(s_n/(\sigma\sqrt{n}))\mathfrak{M}_n$ are asymptotically equivalent. More specifically, we show that
\begin{equation}\label{uco}
 \lim_{n\to\infty}\sup_{t\in[0,1]}\left|\mathfrak{S}_n(t)-\frac{s_n}{\sigma\sqrt{n}}\mathfrak{M}_n(t)\right| = 0.
\end{equation}
Fix $t \in [0,1]$. Using \eqref{sn2} we get that, if $s_k^2\leq ts_n^2<s_{k+1}^2$, then $nt-C_1-1<k\leq nt-C_1$, for some constant $C_1$. Using the boundedness of the increments of $S_n$, we deduce that, for some positive constant $C_2$,
$$S_{\lfloor nt\rfloor}-C_2\leq S_k\leq S_{\lfloor nt\rfloor}+C_2.$$
Finally,
$$\left|\mathfrak{S}_n(t)-\frac{s_n}{\sigma\sqrt{n}}\mathfrak{M}_n(t)\right| = \frac{s_n}{\sigma\sqrt{n}}\left|\frac{S_{\lfloor nt\rfloor}}{s_n}-\frac{S_k}{s_n}\right| \leq \frac{C_2}{\sigma\sqrt{n}}$$
from which we deduce that $\mathfrak{S}_n$ converges weakly to a standard Brownian motion and therefore that $\mathfrak{W}_n$ converges weakly to a $(K+1)$-dimensional Brownian motion.
\end{proof}

\section{Combinatorial proofs}\label{proofs}

\subsection*{Proof of Theorem \ref{2dprob}}

The event $\{X_n=k\}$ is characterised by the number of $(-1)$'s amongst $\{\xi_1,\dots,\xi_n\}$ being equal to $\frac{n-k}{2}$. Call this quantity $n_k$. To further require that $Y_n=l$, the $(+1)$'s and the $(-1)$'s must be arranged in a specific order we describe in the next few lines.

The approach is to think of the $(-1)$'s as defining bins in which the $(+1)$'s must be placed in an appropriate way. Each bin will have a number (possibly zero) of $(+1)$'s followed by one $(-1)$. At the end of this line of bins, we allow one further bin that may only contain $(+1)$'s (or may be empty) -- see Figure \ref{BinsModel}.

\begin{figure}[h]
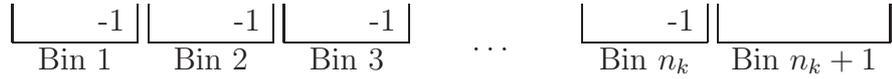

\begin{center}
\begin{tabular}{|ll|}
\hspace*{0.6cm} & -1\\\hline
\multicolumn{2}{c}{Bin 1}
\end{tabular}
\begin{tabular}{|ll|}
\hspace*{0.6cm} & -1\\\hline
\multicolumn{2}{c}{Bin 2}
\end{tabular}
\begin{tabular}{|ll|}
\hspace*{0.6cm} & -1\\\hline
\multicolumn{2}{c}{Bin 3}
\end{tabular}$\qquad\ldots\qquad$
\begin{tabular}{|ll|}
\hspace*{0.6cm} & -1\\\hline
\multicolumn{2}{c}{Bin $n_k$}
\end{tabular}
\begin{tabular}{|ll|}
\hspace*{0.6cm} & \\\hline
\multicolumn{2}{c}{Bin $n_k+1$}
\end{tabular}
\end{center}
\caption{The bins representation}\label{BinsModel}
\end{figure}

To decide on the value of $\eta_{m-1}$, all we need is to identify the bin in which $\xi_m$ falls and more precisely its evenness. Indeed, if $\xi_m$ falls in an even bin, then $\eta_{m-1}=-1$, while if it falls in an odd bin, then $\eta_{m-1}=+1$. The value of $\eta_m$ is simply $\eta_{m-1}\xi_m$.

Let us now denote by $\alpha_i$ the number of $(+1)$ in Bin $2i-1$ (odd bin) and by $\beta_i$ the number of $(+1)$ in Bin $2i$ (even bin). Consider first the case $n_k$ even. Then

$$Y_n = \sum_{i=1}^{n_k/2}(\alpha_i-1)+\alpha_{n_k/2+1}+\sum_{i=1}^{n_k/2}(1-\beta_i) = \sum_{i=1}^{n_k/2+1}\alpha_i-\sum_{i=1}^{n_k/2}\beta_i=\alpha-\beta,$$
where $\alpha$ is the total number of $(+1)$'s in odd bins and $\beta$ the total number of $(+1)$'s in even bins. The requirement that $Y_n=l$ now reduces to the restriction that $\alpha-\beta=l$. Since the total number of $(+1)$'s (in all bins) is $n-n_k$, we deduce that
$$\alpha=\frac{n-n_k+l}2\mbox{ and }\beta=\frac{n-n_k-l}2.$$
In summary, $X_n=k$ and $Y_n=l$ if and only if, amongst $\{\xi_1,\dots,\xi_n\}$, there are $\frac{n+k}2$ $(+1)$'s with $\frac{n-n_k+l}2$ placed in odd bins and $\frac{n-n_k-l}2$ placed in even bins. Therefore, the number of sequences that lead to $X_n=k$ and $Y_n=l$ equals the number of ways of placing $\frac{n-n_k+l}2$ balls into $\frac{n_k}2+1$ (odd) bins and $\frac{n-n_k-l}2$ balls into $\frac{n_k}2$ (even) bins:
$$\mathbb{P}(X_n=k,Y_n=l) = \binom{\frac{n+l}{2}}{\frac{n+k+2l}{4}}\binom{\frac{n-l-2}{2}}{\frac{n+k-2l}{4}}\left(\frac{1}{2}\right)^n.$$
The case $n_k$ odd is dealt with in an identical way.

\subsection*{Proof of Theorem \ref{3dprob}}

\underline{(1) Probability of return to the origin.}

Recall that $W_n = (Y_{0,n},Y_{1,n},Y_{2,n}) \stackrel{d}{=} (Y_{-1,n},Y_{0,n},Y_{1,n})$. Our first task will be to obtain the probability of return to the origin of the process $(Y_{-1,n},Y_{0,n},Y_{1,n})$ in $4n$ steps.

We know from Theorem \ref{2dprob} that $(Y_{0,4n},Y_{1,4n})$ returns to the origin if and only if there are exactly $2n$ $(-1)$'s and $2n$ $(+1)$'s equally split between odd and even bins. To add the constraint that $Y_{-1,4n}$ equals zero, we introduce the concept of a sign change. We shall say that index $i$ ($i\geq1$) represents a sign change if $\xi_i\xi_{i-1}=-1$ ($\xi_0=1$).

Now the event $\{Y_{-1,4n}=0\}$ coincides with the event
$$\left\{\sum_{i=1}^{4n}1_{\mbox{$i$ is a sign change}} = 2n \right\}.$$

For the first bin, no matter whether or not it is empty, there is one sign change. This is because the first $-1$ in the sequence produces a sign change as we suppose that $\xi_0 =1$.
From the second bin to the $(2n)$th bin, each non-empty bin translates into exactly $2$ sign changes. The last bin produces one sign change if it is non-empty and no sign change otherwise.
We see that, in order for $Y_{-1,4n}$ to equal 0, the last bin must be non-empty (i.e. $\xi_{4n}= 1$) and exactly $n-1$ out of bins 2 to $2n$ must be non-empty.

Next we set $\xi_{4n+1}=-1$ and place all digits on a circle thus forming $2n+1$ bins. The extra bin is the one that ends with $\xi_{4n+1}$ (which could also be thought of as $\xi_{-1}$) and is non-empty as $\xi_{4n}=1$.

\begin{figure}[h]
\begin{center}
\begin{tabular}{ll|}
$\cdots$ & 1\ \rd{-1}\\\hline
\multicolumn{2}{c}{$2n+1$}
\end{tabular}
\begin{tabular}{|ll|}
\rd{1} \hspace*{0.4cm} & -1\\\hline
\multicolumn{2}{c}{1}
\end{tabular}
\begin{tabular}{|ll|}
\hspace*{0.4cm} & -1\\\hline
\multicolumn{2}{c}{2}
\end{tabular}
\begin{tabular}{|ll|}
\hspace*{0.4cm} & -1\\\hline
\multicolumn{2}{c}{3}
\end{tabular}$\qquad\ldots\qquad$
\begin{tabular}{|ll|}
\hspace*{0.4cm} & -1\\\hline
\multicolumn{2}{c}{$2n$}
\end{tabular}
\begin{tabular}{|ll}
& $\cdots$\\\hline
\multicolumn{2}{c}{$2n+1$}
\end{tabular}
\end{center}
\caption{The bins representation on a circle}\label{BinsModelCircle}
\end{figure}

With the additional digits, $\xi_0=1$ and $\xi_{4n+1}=-1$, represented in red in Figure \ref{BinsModelCircle}, two additional sign changes are added to the original $2n$.

The scheme now reduces to placing exactly $2n+1$ balls into the bins with the following constraints:
\begin{itemize}
\item bins 1 and $2n+1$ are non-empty (they each have at least one $(+1)$);
\item of the remaining $2n-1$ $(+1)$'s, $n-1$ are placed in odd bins and $n$ in even bins;
\item the number of non-empty bins equals exactly $n+1$ (any non-empty bin translates into two sign changes).
\end{itemize}
The number of non-empty odd bins can be anything from 0 (all $n-1$ $(+1)$'s are in bins 1 and $2n+1$) to a maximum of $n-2$. In fact, it is not possible to have $n-1$ non-empty odd bins as that would imply that all even bins are empty.

Let us now consider the case of $k$ non-empty odd bins.
These must be selected out of $n-1$ odd bins.
The remaining $n-1-k$ non-empty even bins must be selected out of $n$ even bins.
Having selected the non-empty bins, we next count the number of ways to place $n+1$ $(+1)$'s into the $k+2$ odd bins and $n$ $(+1)$'s into the $n-1-k$ even bins (in such a way that all bins are non-empty).
Taking into account these combinatorial observations and summing over the number of non-empty odd bins, we get that the probability of the event of interest equals
\begin{equation}\label{primo}
\sum_{k=0}^{n-2} \binom{n-1}{k} \binom{n}{k+1}\binom{n}{k+1}\binom{n-1}{k+1} 2^{-4n}.
\end{equation}
The case of return to the origin after $4n+2$ steps is obtained in an identical way.

\underline{(2) Order of $\mathbb{P}(W_{2n}=0)$.}

Fix $\alpha \in (1/2,1)$ and let $a_n=\floor{n/2-n^\alpha}$, $b_n=\floor{n/2+n^\alpha}$ and
$$c_{n,k} = \Big(\frac{n-k}{n}\Big)\frac {(n-k)(n-k-1)}{n(k+1)}\Big(\frac{n-k}{k+1}\Big)^2\binom{n}{k}^4$$
so that
$$2^{4n}\mathbb{P}(W_{4n}=0) = \sum_{k\leq a_n}c_{n,k} + \sum_{k\in(a_n, b_n)}c_{n,k} + \sum_{k\geq b_n}c_{n,k} = \Gamma_1 + \Gamma_2 + \Gamma_3.$$
In the sequel, we obtain bounds for each of these three terms. We use $C_i$ to denote various positive constants.

Using Stirling's approximation and the fact that the binomial coefficients are increasing for $k<(n-1)/2$, we get that
\begin{eqnarray*}
\Gamma_1 & \leq & n^3\sum_{k\leq a_n}{n\choose k}^4\ \leq\ n^4{n\choose a_n}^4\ \leq\ C_1n^4\left(\frac{n^n}{(n-a_n)^{n-a_n}a_n^{a_n}}\right)^4\\
& \leq & C_2n^8\left(\frac{n^n}{(n/2-n^\alpha)^{n/2-n^\alpha}(n/2+n^\alpha)^{n/2+n^\alpha}}\right)^4\\
& = & C_2n^8\Big(\frac{n^n(n/2-n^\alpha)^{n^\alpha}}{(n^2/4-n^{2\alpha})^{n/2}(n/2+n^\alpha)^{n^\alpha}}\Big)^4 \\
& = & C_2n^82^{4n}\Big(\frac1{1-4 n^{2\alpha-2}}\Big)^{2n}\Big(\frac{n/2-n^\alpha}{n/2+n^\alpha}\Big)^{4n^\alpha} \\
& = & C_2n^82^{4n}\Big(1+\frac{4n^{2\alpha-2}}{1-4n^{2\alpha-2}}\Big)^{2n}\Big(1-\frac{2n^{\alpha-1}}{1/2+n^{\alpha-1}}\Big)^{4n^\alpha}\\
& \leq & C_3n^82^{4n}\exp\Big(\frac{8n^{2\alpha-1}}{1-4n^{2\alpha-2}}-\frac{16n^{2\alpha-1}}{1+2n^{\alpha-1}}\Big)\\
& \leq & C_3n^82^{4n}\exp\Big(-C_4n^{2\alpha-1}\Big)
\end{eqnarray*}

In the same way, we have
$$\Gamma_3 \leq C_5n^82^{4n}\exp\Big(-C_6n^{2\alpha-1}\Big).$$
Finally we have
$$\Gamma_2 \leq C_7n^\alpha{n\choose n/2}^4 \leq C_82^{4n}n^{\alpha-2}.$$

\underline{(3) Transience.}
As  $n^{\alpha-2} $ is summable ($\alpha-2<-1$), $\sum_n\mathbb{P}(W_{2n}=0)<+\infty$ and $(W_n)_n$ is transient.

\subsection*{Proof of Proposition \ref{nu}}

\underline{(1) For $K=p^\ell$ and $1<n\leq p^\ell$, then $\nu_{K,n}=0$.}

Let $n-1=\alpha_{\ell-1}p^{\ell-1}+\ldots+\alpha_1p+\alpha_0$, where $0\leq \alpha_i<p \mbox{ for } 0\leq i\leq \ell-1$, be the base $p$ expansion of $n-1$. Then the base $p$ expansion of $n+K-2$ is
$$n+K-2=p^\ell+\alpha_{\ell-1}p^{\ell-1}+\ldots+(\alpha_j-1)p^j+(p-1)\sum_{i=0}^{j-1}p^i$$
where $j$ is the first index such that $\alpha_j\neq0$ (i.e. $\alpha_0=\ldots=\alpha_{j-1}=0$ and $\alpha_j\geq 1$).

Since digit $j$ in the base $p$ expansion of $n-1$ (i.e. $\alpha_j$) is greater than digit $j$ in the base $p$ expansion of $n+K-2$ (i.e. $\alpha_j-1$),
we obtain the desired result by application of Lucas Theorem \ref{Lucas}.

\underline{(2) For $n=p^\ell$ and $1\leq K\leq p^\ell$, then $\nu_{K,n}=0$.}

This is an immediate consequence of the symmetry of the array $\nu_{K,n}$.

\underline{(3) For $n=p^\ell+1$ and $1\leq K\leq p^\ell$, then $\nu_{K,n}=1 \mod p$.}

Using (2) we can write
\begin{align*}
\nu_{K,p^{\ell}+1}&=\nu_{K,p^{\ell}}+\nu_{K-1,p^{\ell}+1}=\nu_{K-1,p^{\ell}+1} \mod p \\
&=\nu_{K-1,p^{\ell}}+\nu_{K-2,p^{\ell}+1}=\nu_{K-2,p^{\ell}+1} \mod p \\
&=\dots=\nu_{2,p^{\ell}}+\nu_{1,p^{\ell}+1}=\nu_{1,p^{\ell}+1}=1 \mod p.
\end{align*}

\underline{(4) For $1\leq K\leq p^\ell$, $\nu_{K,p^\ell-K+1}\neq 0 \mod p$.}

When $K=p^{\ell}$, $\nu_{K,p^\ell-K+1}=\nu_{K,1}=1\neq 0$. Let us assume that $1\leq K<p^\ell$. We write $n=p^\ell-K+1$ and $K=\beta_{\ell-1}p^{\ell-1}+\ldots+\beta_1p+\beta_0$ for the base $p$ expansion of $K$. Then the base $p$ expansion of $n-1$ is
\[n-1=p^\ell-K=
\begin{cases}(p-1-\beta_{\ell-1})p^{\ell-1}+\ldots+(p-1-\beta_1)p+(p-\beta_0) & 0<\beta_0<p\\
(p-1-\beta_{\ell-1})p^{\ell-1}+\ldots+(p-1-\beta_j)p^j &\beta_0=0
\end{cases}\]
where $j$ is the first index such that $\beta_j\neq 0$, i.e. $\beta_0=\dots=\beta_{j-1}=0, \beta_j\neq0. $

On the other hand, the base $p$ expansion of $n+K-2$ is:
$$n+K-2=p^\ell-1=(p-1)p^{\ell-1}+\ldots+(p-1)p+(p-1).$$
Again, by application of Lucas Theorem \ref{Lucas}, $\nu_{K,n}$ is not divisible by $p$.


\begin{thebibliography}{8}
\bibitem{Billingsley}
 {\sc Billingsley P.} \ (1995) {\em Probability and Measure. Wiley Series in Probability and Mathematical Statistics}, Wiley New York
\bibitem{Lucas}
 {\sc Lucas E.} \ (1878) {\em Th\'eorie des Fonctions Num\'eriques Simplement P\'eriodiques},
    American Journal of Mathematics, Vol.1 No.3 197-240.
\bibitem{Scott}
 {\sc Scott D.J.} \ (1973) {\em Central Limit Theorems for Martingales and for Processes with Stationary Increments Using a Skorokhod Representation Approach},
    Advances in Applied Probability, Vol.5 No.1(Apr.) 119-137.
\bibitem{Williams}
 {\sc Williams D.} \ (1991) {\em Probability with Martingales. Cambridge University Press.}
\end{thebibliography}
\end{document}